\documentclass[a4paper,10pt]{amsart}

\usepackage[english]{babel}
\usepackage[utf8]{inputenc}

\usepackage{amssymb}
\usepackage{amsmath}
\usepackage{amsthm}
\usepackage{bbm}

\usepackage{enumerate}
\usepackage{tikz}
\usepackage{marginnote}

\newcommand{\bbC}{\mathbb{C}}
\newcommand{\bbN}{\mathbb{N}}

\newcommand{\bbR}{\mathbb{R}}
\newcommand{\bbT}{\mathbb{T}}
\newcommand{\bbZ}{\mathbb{Z}}
\newcommand{\calA}{\mathcal{A}}
\newcommand{\calK}{\mathcal{K}}
\newcommand{\calL}{\mathcal{L}}
\newcommand{\calS}{\mathcal{S}}
\newcommand{\calU}{\mathcal{U}}

\DeclareMathOperator{\id}{\mathit I}

\DeclareMathOperator{\re}{Re}
\DeclareMathOperator{\im}{Im}
\DeclareMathOperator{\dist}{dist}
\newcommand{\distPos}{\operatorname{d}_+}
\DeclareMathOperator{\linSpan}{span}
\newcommand{\phdot}{\mathord{\,\cdot\,}}

\DeclareMathOperator{\spr}{r}

\newcommand{\spec}{\sigma}
\newcommand{\resSet}{\rho}
\newcommand{\Res}{\mathcal{R}}
\newcommand{\per}{{\operatorname{per}}}
\newcommand{\pnt}{{\operatorname{pnt}}}
\newcommand{\appr}{{\operatorname{appr}}}
\newcommand{\ess}{{\operatorname{ess}}}

\newcommand{\Cont}{\mathcal{C}}
\newcommand{\Lp}{L}

\theoremstyle{definition}
\newtheorem{definition}{Definition}[section]
\newtheorem{remark}[definition]{Remark}
\newtheorem{remarks}[definition]{Remarks}
\newtheorem{example}[definition]{Example}
\newtheorem{examples}[definition]{Examples}

\theoremstyle{plain}
\newtheorem{proposition}[definition]{Proposition}
\newtheorem{lemma}[definition]{Lemma}
\newtheorem{theorem}[definition]{Theorem}
\newtheorem{corollary}[definition]{Corollary}
\newtheorem*{theorem_no_number}{Theorem}

\numberwithin{equation}{section}

\begin{document}

\title[The Spectrum of Eventually Positive Operators]{Towards a Perron--Frobenius theory for Eventually Positive Operators}

\author{Jochen Gl\"uck}
\email{jochen.glueck@uni-ulm.de}
\address{Jochen Gl\"uck, Institute of Applied Analysis, Ulm University, 89069 Ulm, Germany}
\keywords{Eventual positivity; asymptotic positivity; Perron--Frobenius theory; Kre\u{\i}n--Rutman theorem; spectral radius; positive eigenvectors; peripheral spectrum}
\subjclass[2010]{47B65; 47A10}

\date{\today}

\begin{abstract}
	This article is a contribution to the spectral theory of so-called eventually positive operators, i.e.\ operators $T$ which may not be positive but whose powers $T^n$ become positive for large enough $n$. While the spectral theory of such operators is well understood in finite dimensions, the infinite dimensional case has received much less attention in the literature.
	
	We show that several sensible notions of ``eventual positivity'' can be defined in the infinite dimensional setting, and in contrast to the finite dimensional case those notions do not in general coincide. We then prove a variety of typical Perron--Frobenius type results: we show that the spectral radius of an eventually positive operator is contained in the spectrum; we give sufficient conditions for the spectral radius to be an eigenvalue admitting a positive eigenvector; and we show that the peripheral spectrum of an eventually positive operator is a cyclic set under quite general assumptions. All our results are formulated for operators on Banach lattices, and many of them do not impose any compactness assumptions on the operator.
\end{abstract}

\maketitle

\section{Introduction}

The classical theorems of Perron and Frobenius about the spectrum of positive matrices, which were published in \cite{Perron1907, Perron1907a} and \cite{Frobenius1908, Frobenius1909, Frobenius1912}, have had a profound impact on mathematical analysis for more than a century now. Given the intrinsic elegance of these theorems as well as their numerous applications (for an overview of several applications, see for instance the survey article \cite{MacCluer2000}) it is no surprise that many attempts were made to generalise those theorems in various respects. Two of those generalisations serve as a motivation for the present paper:

First, it is possible to prove Perron--Frobenius like theorems on infinite dimensional spaces; here, one replaces positive matrices with positive operators defined on a Banach space that is endowed with some kind of ordering. After pioneering work of Kre{\u\i}n and Rutman on rather general ordered Banach spaces \cite{Krein1948} it was noticed later on that positive operators on \emph{Banach lattices} have a particularly rich spectral theory. For an overview of this theory we refer for example to \cite[Section~V.4 and~V.5]{Schaefer1974}, \cite[Chapter~4]{Meyer-Nieberg1991} and \cite{Grobler1995}. Second, one can prove that Perron--Frobenius theorems hold in fact for a much wider class of matrices than only for positive ones. In particular, it is possible to show Perron--Frobenius type results for matrices whose powers become positive for all sufficiently large exponents. Those matrices are usually called \emph{eventually positive}; for the last two decades a very extensive study of them was performed by many researchers (see below for references). 

It might come as surprise that until very recently little effort was made to combine those two approaches, i.e.\ to consider eventually positive operators in infinite dimensions. A step into this direction was made in two recent papers by Daners, Kennedy and the author \cite{Daners2016, Daners2016a} where eventually positive $C_0$-semigroups on infinite dimensional Banach lattices were considered. Spectral theory and Perron--Frobenius type results played an essential role there. Yet, such results were proved under rather strong a priori assumptions on the spectrum in order to obtain \emph{characterisations} of eventual positivity which are well-suited for applications to partial differential equations. In the present paper we go into a different direction: we perform an analysis of the spectral properties of eventually positive operators under very general assumptions. In particular, we do not require much a priori information about the spectrum nor do we need the operators (or their powers) to be compact. 

\subsection*{Contributions of this article} We define several notions of eventual positivity for an operator and we demonstrate by a couple of examples that those notions are not equivalent in infinite dimensions (Section~\ref{section:basic-notions-1-eventual-positivity}). We also define a related but more general notion which we call \emph{asymptotic positivity} in Section~\ref{section:basic-notions-2-asymptotic-positivity}. Again, we give several versions of this notion and show that they are not equivalent in general. The rest of the article is devoted to a thorough spectral analysis of such operators. Motivated by the classical theorems of Perron and Frobenius we discuss three different themes: (i) the question whether the spectral radius of such an operator is contained in the spectrum, (ii) sufficient conditions for the existence of a positive eigenvector for the spectral radius and (iii) symmetry properties of the so-called \emph{peripheral spectrum}, i.e.\ the set of all spectral values of maximal modulus.

Question~(i) is treated in Sections~\ref{section:the-spectral-radius-1} and~\ref{section:the-spectral-radius-2}. Among others we prove the following result (see Corollary~\ref{cor:weakly-ev-pos-implies-spec-in-spec}):

\begin{theorem_no_number}
	Let $T$ be a continuous linear operator on a complex Banach lattice $E$. Suppose that, for all $0 \le x \in E$ and all $0 \le x' \in E'$, there exists an $n_0 \in \bbN$ such that $\langle x', T^n x \rangle \ge 0$ for all $n \ge n_0$. Then the spectral radius of $T$ is contained in the spectrum of $T$.
\end{theorem_no_number}

Sufficient conditions for the spectral radius to be an eigenvalue and to admit a positive eigenvector are given in Section~\ref{section:positive-eigenvectors}. Finally, we deal with symmetry properties of the peripheral (point) spectrum in Sections~\ref{section:the-peripheral-spectrum} and~\ref{section:the-periphera-point-spectrum}; more precisely, we give sufficient conditions for this set to be \emph{cyclic}. Recall that a subset $S$ of the complex numbers is called cyclic if $re^{i\theta} \in S$ (where $r \in [0,\infty)$ and $\theta \in \bbR$) implies that $re^{in\theta} \in S$ for all integers $n \in \bbZ$. One of our results now reads as follows:

\begin{theorem_no_number}
	Let $T$ be a continuous linear operator on a complex Banach lattice, with spectral radius $\spr(T) > 0$. If $T/\spr(T)$ is power bounded and if $T^n \ge 0$ for all sufficiently large $n$, then the peripheral spectrum of $T$ is cyclic.
\end{theorem_no_number}

This result is a simple special case of Theorem~\ref{thm:unif-asympt-pos-adjoint-with-bdd-error-cyclcic-per-spec} below.

\subsection*{On the history of eventual positivity} The literature on eventually positive matrices is vast and we cannot give a complete list of references here. Yet, we want to briefly mention some contributions to the theory. Matrices which have at least \emph{one} positive power were already considered in \cite{Brauer1961}, and \cite[pp 48--54]{Seneta1981} deals with matrices for which some polynomial is positive. Another early paper where eventual positivity occurred is \cite{Friedland1978} where the phenomenon was considered in the context of inverse spectral problems; see also \cite{Zaslavsky1999} for a further paper related to inverse spectral problems. Eventually positive matrices were employed in the study of Perron--Frobenius type properties for matrices with some negative entries in numerous papers such as \cite{Tarazaga2001, Johnson2004, Noutsos2006, Elhashash2008, Elhashash2009}. We refer to \cite{Hogben2009} for an algorithm to determine whether a given matrix is eventually positive and to \cite{CarnochanNaqvi2002, CarnochanNaqvi2004, Hogben2015, Saha2015} for further structure results. Moreover, we also refer to the extensive literature on the relation of eventual positivity and sign patterns of a matrix, see for example \cite{Berman2009, Ellison2009, Catral2012}. We point out once again that these references are far from being complete; the reader can find more information in these articles and in the references therein.

Instead of eventually positive matrices and operators it is also worthwhile studying eventually positive $C_0$-semigroups. In the finite dimensional case such semigroups are, for instance, considered in \cite{Noutsos2008, Olesky2009, Erickson2015} and in \cite[Theorem~2.9]{Ellison2009}. Application to control theory can be found in \cite{Altafini2015, Altafini2016, Sootla2016a}. On infinite dimensional spaces eventually positive $C_0$-semigroups first occurred in the analysis a certain concrete differential operators, namely the bi-Laplace operator \cite{Ferrero2008, Gazzola2008} and the Dirichlet-to-Neumann operator \cite{Daners2014}. The development of a general theory of eventually positive $C_0$-semigroups was recently initiated in \cite{Daners2016, Daners2016a}. It is also worthwhile mentioning the related theory of eventually monotone dynamical systems, see \cite{Sootla2015, Sootla2016}.

\subsection*{Preliminaries} We denote the set of all \emph{strictly} positive integers by $\bbN := \{1,2,...\}$ and we set $\bbN_0 := \bbN \cup \{0\}$. By $\bbT := \{\lambda \in \bbC: \; |\lambda| = 1\}$ we denote the complex unit circle. We recall again that a set $S \subseteq \bbC$ is called \emph{cyclic} if $re^{i\theta} \in S$ ($r \in [0,\infty)$, $\theta \in \bbR$) implies that $re^{in\theta} \in S$ for all $n \in \bbZ$. If $S$ is a subset of a given vector space, then we denote by $\linSpan S$ the \emph{linear span} (or \emph{linear hull}) of $S$. If $S$ is a subset of a metric space $(M,d)$ and if $x \in M$, then we denote by $\dist(x,S) := \inf\{d(x,y): \; y \in S\}$ the \emph{distance} of $x$ to $S$. If $I \subseteq \bbR$ is an interval, then we call a function $\varphi: I \to \bbR$ \emph{increasing} if $\varphi(s) \le \varphi(t)$ for all $s,t \in I$ fulfilling $s \le t$. For every compact Hausdorff space we denote the space of all real- (or complex-) valued continuous functions on $K$ by $\Cont(K;\bbR)$ (respectively, by $\Cont(K;\bbC)$) and we endow this space with the usual supremum norm.

Let $E$ be a real or complex Banach space. We denote the space of all bounded linear operators on $E$ by $\calL(E)$ and the \emph{identity operator} on $E$ by $\id_E$. By $\calK(E)$ we denote the space of all compact linear operators on $E$. The \emph{dual space} of $E$ is denoted by $E'$. For $x \in E$ and $x' \in E'$ we use the common notation $\langle x', x\rangle := x'(x)$. The \emph{adjoint} of an operator $T \in \calL(E)$ is denoted by $T' \in \calL(E')$. An operator $T \in \calL(E)$ is called \emph{power bounded} if $\sup_{n \in \bbN_0} \|T^n\| < \infty$.

Let $E$ be a complex Banach space and let $T \in \calL(E)$. The \emph{spectrum}, the \emph{point spectrum} and the \emph{approximate point spectrum} of $T$ are denoted by $\spec(T)$, $\spec_\pnt(T)$ and $\spec_\appr(T)$, respectively. The number $\spr(T) := \sup \{|\lambda|: \; \lambda \in \spec(T)\}$ denotes the \emph{spectral radius} of $T$ and the two sets
\begin{align*}
	\spec_\per(T) & := \{\lambda \in \spec(T): \; |\lambda| = \spr(T)\} \\
	\text{and} \qquad \spec_{\per,\pnt}(T) & := \{\lambda \in \spec_\pnt(T): \; |\lambda| = \spr(T)\}
\end{align*}
are called the \emph{peripheral spectrum} and the \emph{peripheral point spectrum} of $T$. The number $\spr_\ess(T) := \sup\{|\lambda|: \; \lambda - T \text{ is not Fredholm}\}$ is called the \emph{essential spectral radius} of $T$; it coincides with the spectral radius of the congruence class of $T$ in the Calkin algebra $\calL(E)/\calK(E)$ (this follow from Atkinson's theorem see e.g.\ \cite[Theorem~3.3.2]{Arveson2002}). The \emph{resolvent set} of $T$ is denoted by $\resSet(T) := \bbC \setminus \spec(T)$ and for every $\lambda \in \resSet(T)$ the operator $\Res(\lambda,T) := (\lambda - T)^{-1}$ is called the \emph{resolvent} of $T$ at $\lambda$.

The reader is assumed to be familiar with the theory of real and complex Banach lattices; standard references for this theory are for instance \cite{Schaefer1974} and \cite{Meyer-Nieberg1991}. Let $E$ be a complex Banach lattice. We denote by $E_\bbR$ the underlying real Banach lattice, and by $E_+ := (E_\bbR)_+$ the positive cone. We can decompose every element $x \in E$ as $x = \re x + i\im x$, where $\re x$ and $\im x$ are uniquely determined elements of $E_\bbR$, called the \emph{real part} and the \emph{imaginary part} of $x$. We call a vector $x \in E$ \emph{positive} if $x \in E_+$ and we denote this by $x \ge 0$; moreover, we write $x > 0$ if $x \ge 0$, but $x \not= 0$. More generally, we write $g \ge f$ (respectively, $g > f$) for two elements $f,g \in E$ if $f,g \in E_\bbR$ and if $g-f \ge 0$ (respectively, if $g \ge f$ but $g \not= f$). We call an operator $T \in \calL(E)$ \emph{positive} if $TE_+ \subseteq E_+$ and we denote this by $T \ge 0$. This terminology differs from what is usually used in the theory of (eventually) positive matrices, but it is very common in the theory of Banach lattices. A linear operator $T \in \calL(E)$ is called \emph{real} if $TE_\bbR \subseteq E_\bbR$. Note that every positive operator is real. 

Let $E$ be a real or complex Banach lattice and let $u \in E_+$. The \emph{principal ideal} generated by $u$ is defined to be the set
\begin{align*}
	E_u := \{x \in E: \; \exists c \ge 0 \; |x| \le cu\}
\end{align*}
and for every $x \in E_u$ the \emph{gauge} norm of $x$ with respect to $u$ is defined as
\begin{align*}
	\|x\|_u := \inf \{c \ge 0: \; |x| \le cu\}.
\end{align*}
The gauge norm is indeed a norm on the vector space $E_u$. The space $E_u$ is a (real or complex) Banach lattice with respect to the gauge norm $\|\cdot\|_u$ and with respect to the order inherited from $E$ (or $E_\bbR$, respectively); in fact, $(E_u,\|\cdot\|_u)$ is even an AM-space with unit $u$. These results follow from the corollary to Proposition~II.7.2 in \cite{Schaefer1974}. Kakutani's representation theorem for AM-spaces thus asserts that there exists a compact Hausdorff space $K$ and an isometric Banach lattice isomorphism from $E_u$ to $\Cont(K;\bbR)$ (respectively, to $\Cont(K;\bbC)$) which maps $u$ to the constant function with value $1$ (which we denote by $\mathbbm{1}$). For the case of real scalars this theorem can, for instance, be found in \cite[Theorem~II.7.4]{Schaefer1974} or in \cite[Theorem~2.1.3]{Meyer-Nieberg1991}; the case of complex scalars is a simple consequence of the real case.

A particular role in our paper is played by the distance of a vector to the positive cone. Let $E$ be a complex Banach lattice. For every $f \in E$ we denote by $\distPos(f) := \dist(f,E_+) = \inf\{\|f-g\|: \; g \in E_+\}$ the distance of $f$ to the positive cone $E_+$. This notation applies in particular to the complex Banach lattice $\bbC$ where $\bbC_+ = \bbR_+ = [0,\infty)$. The function $E \to \bbR$, $f \mapsto \distPos(f)$ has a few simple but important properties which we are going to use tacitly throughout: for every $f \in E$ we have $\distPos(f) = 0$ if and only if $f \ge 0$; the function $\distPos(\phdot)$ is continuous (even Lipschitz continuous with Lipschitz constant $1$); we have $\distPos(\alpha f) = \alpha \distPos(f)$ for every $\alpha \in [0,\infty)$ and every $f \in E$; we have $\distPos(f+g) \le \distPos(f) + \distPos(g)$ for all $f,g \in E$; and finally, we have $\distPos(f) \le \|f\|$ for all $f \in E$.

\section{Basic Notions I: Eventual Positivity} \label{section:basic-notions-1-eventual-positivity}

In this section we make precise what me mean by an eventually positive operator. Since we are mainly interested in spectral theory in this paper, we formulate the definition on \emph{complex} Banach lattices.

\begin{definition} \label{def:evtl-pos}
	Let $E$ be a complex Banach lattice. An operator $T \in \calL(E)$ is called
	\begin{itemize}
		\item[(a)] \emph{uniformly eventually positive} if there exists an $n_0 \in \bbN$ such that $T^n \ge 0$ for all $n \ge n_0$. 
		\item[(b)] \emph{individually eventually positive} if for each $x \in E_+$ there exists an $n_0 \in \bbN$ such that $T^nx \ge 0$ for all $n \ge n_0$. 
		\item[(c)] \emph{weakly eventually positive} if for each $x \in E_+$ and each $x' \in E'_+$ there exists an $n_0 \in \bbN_0$ such that $\langle x', T^nx \rangle \ge 0$ for all $n \ge n_0$.
	\end{itemize}
\end{definition}

The definition of uniform and individual eventual positivity was motivated by the papers \cite{Daners2016a} and \cite{Daners2016} where the same terminology was introduced for $C_0$-semigroups; the notion of \emph{weak} eventual positivity seems to be new. 

Obviously, every uniformly eventually positive operator is also individually eventually positive, and every individually eventually positive operator is weakly eventually positive. Moreover, since every finite dimensional complex Banach lattice is isomorphic to $\bbC^d$ (this follows e.g.\ from \cite[Corollary~1 to Theorem~II.3.9]{Schaefer1974}), it is clear that the three notions are in fact equivalent in case that $\dim E < \infty$. Encouraged by this observation one might be tempted to suspect that a similar assertion holds for compact operator, or at least for operators of finite rank, on infinite dimensional Banach lattices. Yet, it turns out that this is not the case as the following two examples show.

\begin{examples} \label{ex:ev-pos-not-equivalent}
	(a) Let $E = \Cont([-1,1];\bbC)$ denote the space of all continuous, complex-valued functions on $[-1,1]$, endowed with the supremum norm $\|\cdot\|_\infty$. There exists an operator $T \in \calL(E)$ which has two-dimensional range and which is individually but not uniformly eventually positive.
	
	(b) Let $p \in [1,\infty)$ and set $E := \Lp^p((-1,1);\bbC)$. There exists an operator $T \in \calL(E)$ which has two-dimensional range and which is weakly but not individually eventually positive.
\end{examples}
\begin{proof}
	We use the following notation: Whenever $E$ is a Banach space, $f \in E$ and $\varphi \in E'$, then we define $f \otimes \varphi \in \calL(E)$ by $(f \otimes \varphi)g = \langle \varphi,g\rangle f$ for all $g \in E$.
	
	(a) Let $E = \Cont([-1,1];\bbC)$. Define $f_1,f_2 \in E$ be $f_1(x) = 1$ and $f_2(x) = x$ for all $x \in [-1,1]$ and define $\varphi_1,\varphi_2 \in E'$ by $\langle \varphi_1, g\rangle = \frac{1}{2} \int_{-1}^1 g(x) \, dx$ and $\langle \varphi_2, g \rangle = \frac{1}{4}[g(1) - g(-1)]$ for all $g \in E$. Note that we have
	\begin{align}
		\label{eq:dualities-for-counterexample}
		\begin{aligned}
			\langle \varphi_1, f_1 \rangle = 1, & \qquad \langle \varphi_1, f_2 \rangle = 0,\\
			\langle \varphi_2, f_1 \rangle = 0, & \qquad \langle \varphi_2, f_2 \rangle = \frac{1}{2}.
		\end{aligned}
	\end{align}
	We define the operator $T \in \calL(E)$ by $T = f_1 \otimes \varphi_1 + f_2 \otimes \varphi_2$. Obviously the range $TE$ of $T$ has dimension $2$.
	
	We have $T^n = f_1 \otimes \varphi_1 + \frac{1}{2^{n-1}} f_2 \otimes \varphi_2$ for all $n \in \bbN$. Let us show that $T$ is individually eventually positive: for every $0 < g \in E$ we have
	\begin{align*}
		T^ng = \langle \varphi_1,g \rangle f_1 + \frac{1}{2^{n-1}} \langle \varphi_2, g \rangle \to \langle \varphi_1, g \rangle f_1
	\end{align*}
	as $n \to \infty$. Since $\langle \varphi_1, g \rangle > 0$, since $f_1$ is the constant function with value $1$ and since $T^ng$ is real-valued for each $n$, it follows that $T^ng \ge 0$ for all sufficiently large $n$; hence, $T$ is individually eventually positive.
	
	Now we show that $T$ is not uniformly eventually positive. For every $\varepsilon > 0$ we choose a function $0 \le g_\varepsilon \in E$ which fulfils $\int_{-1}^1 g_\varepsilon(x) \, dx \le \varepsilon$, $g_\varepsilon(-1) = 1$ and $g_\varepsilon(1) = 0$. Then we obtain for every $n \in \bbN$ that
	\begin{align*}
		T^ng_\varepsilon = \frac{1}{2} \int_{-1}^1 g_\varepsilon(x) \, dx f_1 - \frac{1}{2^{n-1}} \frac{1}{4} f_2 \le \frac{1}{2} \varepsilon f_1 - \frac{1}{2^{n+1}} f_2
	\end{align*}
	and thus,
	\begin{align*}
		(T^ng_\varepsilon)(-1) \le \frac{1}{2} \varepsilon - \frac{1}{2^{n+1}}.
	\end{align*}
	For every $n \in \bbN$ we can therefore find an $\varepsilon > 0$ and a function $g_\varepsilon \ge 0$ such that $(T^ng_\varepsilon)(-1) < 0$. This even proves that $T^n$ is not a positive operator for any $n \in \bbN$. In particular, $T$ is not uniformly eventually positive.
	
	(b) Now, let $p \in [1,\infty)$ and $E = \Lp^p((-1,1);\bbC)$. Let $f_1,f_2 \in E$ be given by $f_1(x) = 1$ and $f_2(x) = \operatorname{sgn}x \, |x|^{-\frac{1}{2p}}$ for all $x \in (-1,1)$; note that $f_2$ is indeed contained in $\Lp^p$ due to the choice of the exponent $-\frac{1}{2p}$. Define $\varphi_1,\varphi_2 \in E'$ by $\langle \varphi_1, g \rangle = \frac{1}{2} \int_{(-1,1)} g(x) \, dx$ and $\langle \varphi_2, g\rangle = c\int_{(-1,1)} \operatorname{sgn}x \, g(x) \, dx$ for all $g \in E$, where the constant $c > 0$ is chosen such that $\langle \varphi_2, f_2 \rangle = \frac{1}{2}$. Note that $f_1,f_2$ and $\varphi_1, \varphi_2$ fulfil the equations~\eqref{eq:dualities-for-counterexample}.
	
	As above we define $T := f_1 \otimes \varphi_1 + f_2 \otimes \varphi_2$ and we thus obtain
	\begin{align*}
		T^n = f_1 \otimes \varphi_1 + \frac{1}{2^{n-1}} f_2 \otimes \varphi_2
	\end{align*}
	for all $n \in \bbN$. Let us show that $T$ is weakly eventually positive. For every $0 < g \in E$ and and every $0 < \psi \in E'$ we obtain
	\begin{align*}
		\langle \psi, T^n g\rangle = \langle \varphi_1, g\rangle \langle \psi, f_1\rangle + \frac{1}{2^{n-1}} \langle \varphi_2, g\rangle \langle \psi, f_2 \rangle \to \langle \varphi_1, g \rangle \langle \psi, f_1 \rangle > 0
	\end{align*}
	as $n \to \infty$. Hence, $\langle \psi, T^n g \rangle$ is positive for all sufficiently large $n$.
	
	Now we show that $T$ is not individually eventually positive. To this end, choose a function $0 < g \in E$ such that $\langle \varphi_2, g \rangle > 0$. We obtain for all $n \in \bbN$ that
	\begin{align*}
		T^ng = \langle \varphi_1, g\rangle f_1 + \frac{1}{2^{n-1}} \langle \varphi_2, g\rangle f_2.
	\end{align*}
	The function $f_1$ is bounded, but the function $f_2$ fulfils $\lim_{x \uparrow 0}f_2(x) = -\infty$. This proves that $T^ng \not \ge 0$ for any $n \in \bbN$. In particular, $T$ is not individually eventually positive.
\end{proof}

We note in passing that one can also construct examples of $C_0$-semigroups which are, say, individually but not uniformly eventually positive (see \cite[Examples~5.7 and~5.8]{Daners2016}); those examples are, however, a bit more involved, in particular if one wants to ensure certain compactness properties.

\section{Basic Notions II: Asymptotic Positivity} \label{section:basic-notions-2-asymptotic-positivity}

The second notion we deal with in this paper is \emph{asymptotic positivity}. For certain $C_0$-semigroups two versions of this notion were introduced in \cite{Daners2016a}. Here we define three versions of the notion for (powers of) single operators. Recall that for every element $f$ of a complex Banach lattice $E$ the symbol $\distPos(f) := \dist(f,E_+)$ denotes the distance of $f$ to the positive cone.

\begin{definition} \label{def:asymp-pos}
	Let $T$ be a bounded linear operator on a complex Banach lattice $E$. Let $r(T) > 0$ and define $S := T/\spr(T)$. We call $T$
	\begin{itemize}
		\item[(a)] \emph{uniformly asymptotically positive} $\sup_{x \in E_+, \; \|x\| \le 1} \distPos(S^n x) \to 0$ as $n \to \infty$.
		\item[(b)] \emph{individually asymptotically positive} if $\distPos(S^nx) \to 0$ as $n \to \infty$ for each $x \in E_+$.
		\item[(c)] \emph{weakly asymptotically positive} if $\distPos(\langle x', S^nx \rangle) \to 0$ as $n \to \infty$ for each $x \in E_+$ and each $x' \in E'_+$.
	\end{itemize}
\end{definition}

\begin{remarks}
	(a) The above definition is rather bold in the following sense. Consider an operator $T$ on a complex Banach lattice $E$ with $r(T) = 1$ and assume that $T$ is uniformly asymptotically positive. In case that $T$ is power bounded uniform asymptotic positivity means that the fraction
	\begin{align}
		\label{eq:fraction-for-uniform-asympt-pos}
		\frac{\sup_{x \in E_+, \; \|x\| = 1} \distPos(T^n x)}{\|T^n\|}
	\end{align}
	converges to $0$ as $n \to \infty$. If, however, $T$ is not power-bounded, then the condition that $T$ be uniformly asymptotically positive is \emph{stronger} than the condition that~\eqref{eq:fraction-for-uniform-asympt-pos} converge to $0$. 
	
	Hence, one could also suggest to call $T$ asymptotically positive if only the fraction~\eqref{eq:fraction-for-uniform-asympt-pos} converges to $0$, and it does not seem to be clear whether one should prefer this definition or the definition that we gave above (and that we use throughout the paper). This is the reason why asymptotic positivity for $C_0$-semigroups was only defined under additional boundedness assumptions in \cite{Daners2016a} (compare also \cite[Problem~(c) in Section~10]{Daners2016a}).
	
	(b) We did not assume the operator $T$ in Definition~\ref{def:asymp-pos} to be real. A simple example for a non-real but uniformly asymptotically positive operator $T$ on the complex Banach lattice $\bbC^2$ is given by
	\begin{align*}
		T = 
		\begin{pmatrix}
			1 & 0  \\
			0 & \frac{i}{2}
		\end{pmatrix}.
	\end{align*}
	None of our subsequent spectral results for asymptotically positive operators requires the operator to be real. Hence, (the finite dimensional versions of) these results contribute to the Perron--Frobenius theory of matrices with some complex entries, a topic which has already been studied by several authors and from various perspectives (see for instance \cite{Rump2003, Noutsos2012, Tudisco2015}).
\end{remarks}

The following proposition might help to get a better understanding of the distance to the positive cone which plays an important role in Definition~\ref{def:asymp-pos}.

\begin{proposition} \label{prop:formula-for-distance-to-positive-cone}
	Let $E$ be a complex Banach lattice and let $x \in E$. Then we have $\distPos(x) = \|x - (\re x)^+\| = \|-(\re x)^- + i \im x\|$.
\end{proposition}
\begin{proof}
	Let $x \in E$. The second equality in the assertion is obvious, so let us prove the first equality. We define $\hat x := x - (\re x)^+$. Of course we have $\distPos(x) \le \|\hat x\|$.
	
	In order to prove the converse inequality it suffices to show that $|\hat x| \le |x - y|$ for all $y \in E_+$. This inequality is easy to check in case that $E = \bbC$ and hence it is also true if $E$ is the space $\Cont(K;\bbC)$ of continuous complex-valued functions on a compact Hausdorff space $K$. Now, let $E$ be arbitrary and let $y \in E_+$. We set $u := |x| + y \in E_+$. The principal ideal $E_u$ contains both $x$ and $y$ and, when endowed with the gauge norm $\|\cdot\|_u$, it is a complex Banach lattice which is isometrically Banach lattice isomorphic to a $\Cont(K;\bbC)$-space. Therefore, the inequality $|\hat x| \le |x-y|$ is true in the complex Banach lattice $(E_u,\|\cdot\|_u)$. Since $E_u$ is an ideal in $E$, the complex modulus in $E_u$ coincides with the complex modulus in $E$ for every element of $E_u$. Thus we have $|\hat x| \le |x-y|$ in $E$, as claimed.
\end{proof}

Let us briefly comment on the relation between the three notions in Definition~\ref{def:asymp-pos}:

\begin{proposition} \label{prop:asymp-pos-relation-between-different-versions}
	Let $E$ be a complex Banach lattice and let $T \in \calL(E)$ be an operator with $\spr(T)$.
	\begin{itemize}
		\item[(a)] If $T$ is uniformly asymptotically positive, then $T$ is also individually asymptotically positive.
		\item[(b)] If $T$ is individually asymptotically positive, then $T$ is also weakly asymptotically positive.
	\end{itemize}
\end{proposition}
\begin{proof}
	The implication in~(a) is obvious and the implication in~(b) follows from the estimate
	\begin{align*}
		\distPos(\langle x', y\rangle) \le \inf_{x \in E_+} |\langle x', y \rangle - \langle x',x \rangle| \le \|x'\| \distPos(y),
	\end{align*}
	which is true for all vectors $y \in E$ and for all functionals $0 \le x' \in E'$.
\end{proof}

As in Section~\ref{section:basic-notions-1-eventual-positivity}, the converse implications are not in general true. We demonstrate this by two simple examples:

\begin{examples} \label{ex:asymp-pos-not-equivalent}
	Let $p \in [1,\infty)$ and let $E := \ell^p(\bbN;\bbC)$.
	
	(a) There exists an operator $T \in \calL(E)$ with spectral radius $\spr(T) = 1$ which has the following properties: $T$ is not uniformly asymptotically positive, but $T^n$ converges strongly to $0$ as $n \to \infty$ and thus $T$ is individually asymptotically positive (though for trivial reasons).
	
	(b) There exists an operator $T \in \calL(E)$ with spectral radius $\spr(T) = 1$ which has the following properties: $T$ is not individually asymptotically positive, but $T^n$ converges weakly to $0$ as $n \to \infty$ and thus $T$ is weakly asymptotically positive (though for trivial reasons). 
\end{examples}
\begin{proof}
	(a) If we define $T$ to be the multiplication operator with symbol $(-1+\frac{1}{n})_{n \in \bbN}$, then $T$ clearly has all the claimed properties.
	
	(b) Let $T$ be $-1$ times the right shift operator on $E$; then $T$ has all the properties that we claimed.
\end{proof}

On the other hand weak, individual, and even uniform eventual positivity coincide under sufficiently strong compactness assumptions as we prove in the following proposition. 

\begin{proposition} \label{prop:regularity-implies-better-asymp-pos}
	Let $E$ be a complex Banach lattice and let $T \in \calL(E)$ be an operator with $\spr(T) > 0$ which is weakly asymptotically positive. Denote by $\calS := \{(T/\spr(T))^n: \; n \in \bbN_0\}$ the semigroup in $\calL(E)$ generated by the rescaled operator $T/\spr(T)$.
	\begin{itemize}
		\item[(a)] If $\calS$ is relatively compact in $\calL(E)$ with respect to the strong operator topology, then $T$ is individually asymptotically positive.
		\item[(b)] If $\calS$ is relatively compact in $\calL(E)$ with respect to the operator norm topology, then $T$ is uniformly asymptotically positive.
	\end{itemize}
\end{proposition}

In the situation of the above proposition, define $S := T/\spr(T)$. If the set $\calS$ is relatively compact with respect to the strong operator topology, then $S$ is usually said to have \emph{relatively compact orbits} or to be \emph{almost periodic} in the literature. It is well-known that $\calS = \{S^n: \; n \in \bbN_0\}$ is relatively compact in $\calL(E)$ with respect to the strong operator topology if and only if, for every $f \in E$, the set $\{S^nf: \; n \in \bbN_0\}$ is relatively compact in $E$ with respect to the norm topology; see e.g.~\cite[Corollary~A.5]{Engel2000}.

The condition in assertion~(b) that $\calS$ be relatively compact with respect to the operator norm topology is for instance fulfilled if $S$ is power-bounded and some power of $T$ is compact.

\begin{proof}[Proof of Proposition~\ref{prop:regularity-implies-better-asymp-pos}]
	We may assume throughout the proof that $\spr(T) = 1$.
	
	(a) Let $f \in E_+$. We have to show that $\distPos(T^nf)$ converges to $0$ as $n \to \infty$, and to this end it suffices to prove that every subsequence $(\distPos(T^{n_k}f))_{k \in \bbN_0}$ of $(\distPos(T^nf))_{n \in \bbN_0}$ has itself a subsequence which converges to $0$. Since the set $\{T^{n_k}f: \; k \in \bbN_0\}$ is relatively compact in $E$, the sequence $(T^{n_k}f)_{k \in \bbN_0}$ has a subsequence $(T^{n_{k_j}})_{j \in \bbN_0}$ which converges to a vector $g \in E$. Since $T$ is weakly asymptotically positive, one readily obtains that $g \in E_+$. Hence,
	\begin{align*}
		\distPos(T^{n_{k_j}}f) \le \|T^{n_{k_j}} - g\| \to 0,
	\end{align*}
	which proves that claim.
	
	(b) Let $\calL(E)_+$ denote the set of all positive operators in $\calL(E)$. If the assumption of~(b) is fulfilled, then we can show by the same arguments as in~(a) that $\dist(T^n,\calL(E)_+) \to 0$ as $n \to \infty$. Now, let $\varepsilon > 0$ and choose $n_0 \ge n$ such that $\dist(T^n,\calL(E)_+) < \varepsilon$ for each $n \ge n_0$. Then, for every $n \ge n_0$, we can find an operator $R_n \ge 0$ such that $\|T^n-R_n\| < \varepsilon$. In particular we obtain for every $f \in E_+$ with $\|f\| \le 1$ and every $n \ge n_0$ that
	\begin{align*}
		\distPos(T^nf) \le \|T^nf - R_nf\| + \distPos(R_nf) \le \|T^n-R_n\| < \varepsilon.
	\end{align*}
	Thus,
	\begin{align*}
		\sup_{f \in E_+, \; \|f\| \le 1} \distPos(T^nf) \le \varepsilon
	\end{align*}
	for all $n \ge n_0$, which proves that $T$ is indeed uniformly asymptotically positive.
\end{proof}

\section{The Spectral Radius I} \label{section:the-spectral-radius-1}

Let $E$ be a complex Banach lattice and let $T \in \calL(E)$. If $T$ is positive, then it is well-known that $\spr(T) \in \spec(T)$, see e.g.\ \cite[Proposition~V.4.1]{Schaefer1974}. In this section we prove that the same is still true for uniformly asymptotically positive operators. Individually asymptotically positive operator, for which the situation is more subtle, and individually eventually positive operators are treated in the next section. Our main result in this section is as follows:

\begin{theorem} \label{thm:spec-rad-in-spec-for-unif-asymp-pos-op}
	Let $E$ be a complex Banach lattice and let $T \in \calL(E)$ an operator with $\spr(T) > 0$. If $T$ is uniformly asymptotically positive, then $r(T) \in \sigma(T)$.
\end{theorem}

The proof of the above theorem is most easily understood if we recall how the proof for positive operators works. So let $T$ be a positive operator on a complex Banach lattice $E$. A simple application of the Neumann series representation of the resolvent yields the estimate $|\Res(\lambda,A)f| \le \Res(|\lambda|,A)|f|$ for all $f \in E$ and all $\lambda \in \bbC$ with $|\lambda| > \spr(T)$. From this resolvent estimate one can easily deduce that $\spr(T) \in \spec(T)$. 

For the proof of Theorem~\ref{thm:spec-rad-in-spec-for-unif-asymp-pos-op} we use a similar approach. Let us begin by showing a version of the estimate $|\Res(\lambda,A)f| \le \Res(|\lambda|,A)|f|$; since the positivity is only asymptotic now, a certain error term occurs (compare also \cite[Lemma~7.4]{Daners2016}):

\begin{lemma} \label{lem:resolvent-estimate-for-unif-asymp-pos-op}
	Let $E$ be a complex Banach lattice and let $T \in \calL(E)$ be an operator with $\spr(T) = 1$ which is uniformly asymptotically positive. Then there is a function $\omega: (1,\infty)\times E_+ \to E_+$ with the following properties:
	\begin{itemize}
		\item[(a)] For each $x \in E_+$ and each $\lambda \in \bbC$ satisfying $|\lambda| > 1$ we have
			\begin{align*}
				|\Res(\lambda,T)x| \le \re\big(\Res(|\lambda|,T)x\big) + \omega(|\lambda|,x) \text{.}
			\end{align*}
		\item[(b)] We have $\sup_{x \in E_+,\; \|x\| \le 1} (r - 1)\|\omega(r,x)\| \to 0$ as $r \downarrow 1$.
	\end{itemize}
\end{lemma}
\begin{proof}
	For each $r > 1$ and each $x \in E_+$ we define
	\begin{align*}
		\omega(r,x) = \sum_{n=0}^\infty \frac{1}{r^{n+1}} \big( |T^nx| - \re (T^n x) \big).
	\end{align*}
	Let us show that this function fulfils the assertions~(a) and~(b).
	
	(a)	For every $x \in E_+$ and every $\lambda \in \bbC$, $|\lambda| > 1$, we have
	\begin{align*}
		\omega(|\lambda|,x) & = \sum_{n=0}^\infty |\frac{1}{\lambda^{n+1}} T^n x| - \sum_{n=0}^\infty \frac{1}{|\lambda|^{n+1}} \re(T^n x) \ge \\
		& \ge |\sum_{n=0}^\infty \frac{1}{\lambda^{n+1}} T^n x| - \re \sum_{n=0}^\infty \frac{1}{|\lambda|^{n+1}} T^n x.
	\end{align*}
	The first summand in the latter term equals $|\Res(\lambda,T)x|$ and the second summand equals $\re \big(\Res(|\lambda|,T)x\big)$, so we obtain (a).
	
	(b) Define $\delta_n := \sup_{x \in E_+, \; \|x\| \le 1} \distPos(T^n x)$ for each $n \in \bbN_0$. Since $T$ is uniformly asymptotically positive, we have $\delta_n \to 0$ as $n \to \infty$. As explained in the subsequent Remark~\ref{rem:estimate-for-real-part-and-distance-to-positive-cone} we have $\| |y| - \re y \| \le 2 \distPos(y)$ for each $y \in E$. Using this, we obtain that
	\begin{align*}
		\sup_{x \in E_+, \; \|x\| \le 1} \|\omega(r,x)\| \le \sup_{x \in E_+, \; \|x\| = 1} 2 \sum_{n=0}^\infty \frac{1}{r^{n+1}} \dist(T^nx,E_+) \le 2 \sum_{n=0}^\infty \frac{1}{r^{n+1}} \delta_n \text{,}
	\end{align*}
	for every $r > 1$. 
	
	Using that $\sum_{n=0}^\infty \frac{r-1}{r^{n+1}} = 1$ for all $r > 1$ and that $\delta_n \to 0$ as $n \to \infty$, we can see that $\sum_{n=0}^\infty \frac{r-1}{r^{n+1}}\delta_n \to 0$ as $r \downarrow 1$. Hence, we obtain~(b).
\end{proof}

In the proof of Lemma~\ref{lem:resolvent-estimate-for-unif-asymp-pos-op} we made use of the following observation.

\begin{remark} \label{rem:estimate-for-real-part-and-distance-to-positive-cone}
	Let $E$ be a complex Banach lattice. Then we have $\| |x| - \re x \| \le 2 \distPos(x)$ for every $x \in E$.
\end{remark}
\begin{proof}
	According to Proposition~\ref{prop:formula-for-distance-to-positive-cone} we have have $\|x - (\re x)^+\| = \distPos(x)$, so it suffices to show that $0 \le |x| - \re x \le 2|x - (\re x)^+|$. 
	
	The first equality is obvious since we have $\re x \le |\re x| \le |x|$. In order to prove the second inequality $|x| - \re x \le 2|x - (\re x)^+|$ one argues as in the proof of Proposition~\ref{prop:formula-for-distance-to-positive-cone}: first one checks by a brief computation that the inequality holds if $E = \bbC$ and hence it also holds if $E$ is the space of continuous complex-valued functions on any compact Hausdorff space. This implies that the inequality is true in the principal ideal $E_{|x|}$ and hence in $E$.
\end{proof}

The last ingredient that we need for the proof of Theorem~\ref{thm:spec-rad-in-spec-for-unif-asymp-pos-op} is the following simple observation about the norm of operators on a complex Banach lattice.

\begin{remark} \label{rem:norm-estimate-on-positive-cone}
	Let $E$ be a complex Banach lattice and let $T \in \calL(E)$. Then there exists a vector $x \in E_+$ of norm $\|x\| \le 1$ such that $\|Tx\| \ge \frac{1}{8}\|T\|$.
\end{remark}
\begin{proof}
	The assertion is obvious of $T = 0$, so assume that $\|T\| > 0$. Then we can find a vector $z \in E$ of norm $\|z\| \le 1$ such that $\frac{1}{2}\|T\| \le \|Tz\|$ and hence,
	\begin{align*}
		\frac{1}{2}\|Tz\| \le \|T(\re z)^+\| + \|T(\re z)^-\| + \|T(\im z)^+\| + \|T(\im z)^-\|.
	\end{align*}
	Thus, at least one the latter four summands is $\ge \frac{1}{8}\|Tz\|$.
\end{proof}

The estimate in the above remark is of course not optimal, but it suffices for our purposes. Using the resolvent estimate from Lemma~\ref{lem:resolvent-estimate-for-unif-asymp-pos-op} we can now prove Theorem~\ref{thm:spec-rad-in-spec-for-unif-asymp-pos-op}:

\begin{proof}[Proof of Theorem~\ref{thm:spec-rad-in-spec-for-unif-asymp-pos-op}]
	We may assume that $\spr(T) = 1$. Let $\lambda \in \bbC$ be a spectral value of $T$ of modulus $1$ and choose a sequence $(r_n) \subseteq (1,\infty)$ which converges to $1$. According to Remark~\ref{rem:norm-estimate-on-positive-cone} we can find a sequence of vectors $(x_n) \subseteq E_+$, each of them of norm $\le 1$, such that
	\begin{align*}
		\|\Res(r_n\lambda,T)x_n\| \ge \frac{1}{8} \|\Res(r_n\lambda,T)\| \ge \frac{1}{8} \frac{1}{\dist(r_n\lambda,\spec(T))} = \frac{1}{8(r_n-1)}.
	\end{align*}
	Let $\omega: (1,\infty) \times E_+ \to E_+$ be as in Lemma~\ref{lem:resolvent-estimate-for-unif-asymp-pos-op}. For each index $n$ we define $\delta_n := \sup_{x \in E_+, \, \|x\| \le 1} \|\omega(r_n,x)\|$. Then
	\begin{align*}
		(r_n-1) & \|\Res(r_n,T)\| \ge (r_n-1)\|\re (\Res(r_n,T)x_n)\| \\
		& \ge (r_n-1)\|\Res(r_n\lambda,T)x_n\| - (r_n-1)\delta_n \ge \frac{1}{8} - (r_n-1)\delta_n \to \frac{1}{8}
	\end{align*}
	as $n \to \infty$, so $\|\Res(r_n,T)\| \to \infty$. Since $(r_n)$ converges to $1$ we conclude that $1 \in \spec(T)$, as claimed.
\end{proof}

\section{The Spectral Radius II} \label{section:the-spectral-radius-2}

While the spectral radius of a uniformly asymptotically positive operator is always contained in its spectrum according to Theorem~\ref{thm:spec-rad-in-spec-for-unif-asymp-pos-op}, the situation is more subtle for individually and weakly asymptotically positive operators. We first demonstrate by a simple example what is \emph{not} true:

\begin{example}
	Let $p \in [1,\infty)$ and let $E := \ell^p(\bbN;\bbC)$. There exists an operator $T \in \calL(E)$ with spectral radius $\spr(T) = 1$ which has the following properties: the powers $T^n$ converges strongly to $0$ as $n \to \infty$, so $T$ is individually asymptotically positive; yet, the spectral radius $\spr(T) = 1$ is not contained in the spectrum $\spec(T)$.
\end{example}
\begin{proof}
	Let $T$ be the multiplication operator with symbol $(-1+\frac{1}{n})_{n \in \bbN}$ (which we have already considered in Example~\ref{ex:asymp-pos-not-equivalent}(a)). Then $T$ fulfils all the properties we claimed.
\end{proof}

The above example raises the question whether at least for weakly \emph{eventually} positive operators the spectral radius is contained in the spectrum. This is indeed the case, and it follows from a more general result: we will see in Theorems~\ref{thm:countable-summability-condition-implies-spec-in-spec} and~\ref{thm:weakly-asymp-pos-with-rate-implies-spec-rad-in-spec} below that the spectral radius of a weakly asymptotically positive operator is automatically contained in the spectrum provided that the sequences $\big(\distPos(\langle x', T^nx \rangle)\big)_{n \in \bbN_0}$ (for $x \in E_+$, $x' \in E'_+$) do not only converge to $0$, but satisfy a certain decay rate.

Such a result might not come as a complete surprise and it is motivated by the following observation: let $T$ be a continuous linear operator on a, say complex, Banach space $E$. If, for all $x \in E$ and all $x' \in E'$, the sequence $(\langle x', T^nx\rangle)_{n \in \bbN_0}$ converges to $0$ \emph{with a certain rate}, then the powers $T^n$ actually converge to $0$ with respect to the operator norm; results of this type can for instance be found in \cite{Weiss1989}, \cite{Neerven1995} and \cite{Glueck2015}. Here, we consider an operator $T$ on a complex Banach lattice $E$ such that for all $x,x'\ge 0$ the sequence $\big((\distPos(\langle x',T^nx\rangle)\big)_{n \in \bbN_0}$ has a certain decay rate. We do not know whether this condition already implies that $T$ is uniformly asymptotically positive, but we are going to show that the condition implies $\spr(T) \in \spec(T)$. Our main results in this section are the following theorem and its corollaries:

\begin{theorem} \label{thm:countable-summability-condition-implies-spec-in-spec}
	Let $E$ be a complex Banach lattice, let $T \in \calL(E)$ with $r(T) > 0$ and define $S := T/\spr(T)$. Let $\Phi$ be an at most countable set of increasing functions $\varphi: [0,\infty) \to [0,\infty)$ which fulfil $\varphi(t) > 0$ for all $t > 0$. Suppose that for every $x \in E_+$ with $\|x\| \le 1$ and for every $x \in E'_+$ with $\|x'\| \le 1$ there exists a function $\varphi \in \Phi$ such that
	\begin{align*}
		\sum_{n=0}^\infty \varphi\big(\distPos(\langle x', S^n x \rangle)\big) < \infty \text{.}
	\end{align*}
	Then $\spr(T) \in \spec(T)$.
\end{theorem}

It follows readily from Theorem~\ref{thm:countable-summability-condition-implies-spec-in-spec} that we have $\spr(T) \in \spec(T)$ in case that the distance of $\langle x', S^n x \rangle$ to the positive real numbers is summable. Let us formulate this as an extra corollary:

\begin{corollary} \label{cor:negative-parts-summable-implies-spec-in-spec}
	Let $E$ be a complex Banach lattice, let $T \in \calL(E)$ with $r(T) > 0$ and define $S := T/\spr(T)$. Suppose that
	\begin{align*}
		\sum_{n=0}^\infty \distPos(\langle x', S^n x \rangle) < \infty
	\end{align*}
	for all $x \in E_+$, $x' \in E'_+$. Then $\spr(T) \in \spec(T)$.
\end{corollary}

From Theorem~\ref{thm:countable-summability-condition-implies-spec-in-spec} (or from Corollary~\ref{cor:negative-parts-summable-implies-spec-in-spec}) it follows, in particular, that the spectral radius of a weakly eventually positive operators is contained in the spectrum. We state this in an extra corollary, too:

\begin{corollary} \label{cor:weakly-ev-pos-implies-spec-in-spec}
	Let $E \not= \{0\}$ be a complex Banach lattice and let $T \in \calL(E)$ be weakly eventually positive. Then $\spr(T) \in \spec(T)$.
\end{corollary}

Finally, we formulate another consequence of Theorem~\ref{thm:countable-summability-condition-implies-spec-in-spec} which is more general then Corollary~\ref{cor:negative-parts-summable-implies-spec-in-spec}:

\begin{corollary} \label{cor:negative-parts-in-l_p-implies-spec-in-spec}
	Let $E$ be a complex Banach lattice, let $T \in \calL(E)$ with $r(T) > 0$ and define $S := T/\spr(T)$. Suppose that
	\begin{align*}
		\big(\distPos(\langle x', S^n x \rangle)\big)_{n \in \bbN_0} \in \bigcup_{1 \le p < \infty} \ell^p(\bbN_0; \bbR)
	\end{align*}
	for all $x \in E_+$, $x' \in E'_+$. Then $\spr(T) \in \spec(T)$.
\end{corollary}

For the proof of Theorem~\ref{thm:countable-summability-condition-implies-spec-in-spec} we employ techniques from \cite{Glueck2015}. We first need to introduce a bit of terminology. By $c_0$ we denote the space of all complex-valued sequences which are indexed over $\bbN_0$ and which converge to $0$. We endow this space with the supremum norm which renders it a complex Banach lattice. In particular, for every $u \in (c_0)_+$, the principal ideal $(c_0)_u$ is defined as explained in the preliminaries.

For a sequence $0 \le x = (x_n)_{n \in \bbN_0} \in c_0$ we denote by $x^* := (x_n^*)_{n \in \bbN_0} \in c_0$ the \emph{decreasing rearrangement} of $x$, i.e.~the sequence consisting of the same entries as $x$ (including multiplicities) which have been rearranged in decreasing order. The following definition is taken from \cite[Definition~3.1]{Glueck2015}.

\begin{definition}
	Let $F \subset (c_0)_+$ and let $a = (a_n)_{n \in \bbN_0}$ be a sequence of complex numbers. We say that $F$ \emph{governs} the sequence $a$ if $a \in c_0$ and if there exists an element $f \in F$ such that the decreasing rearrangement $|a|^*$ of $|a|$ is contained in the principal ideal $(c_0)_f$.
\end{definition}

A corollary of the following quite general result will be the key to give the proof of Theorem~\ref{thm:countable-summability-condition-implies-spec-in-spec} at the end of the section.

\begin{theorem} \label{thm:weakly-asymp-pos-with-rate-implies-spec-rad-in-spec}
	Let $E$ be a complex Banach lattice, let $T \in \calL(E)$ with $\spr(T) > 0$ and define $S := T/\spr(T)$. Let $f \in (c_0)_+$ and suppose that $\{f\}$ governs the sequence
	\begin{align*}
		\big(\distPos(\langle x', S^n x \rangle)\big)_{n \in \bbN_0}
	\end{align*}
	for each $x \in E_+$ and each $x \in E'_+$. Then $r(T) \in \sigma(T)$.
\end{theorem}

We can prove Theorem~\ref{thm:weakly-asymp-pos-with-rate-implies-spec-rad-in-spec} by a method which was also employed in the proof of \cite[Theorem~3.2]{Glueck2015}:

\begin{proof}[Proof of Theorem~\ref{thm:weakly-asymp-pos-with-rate-implies-spec-rad-in-spec}]
	We may assume that $\spr(T) = 1$ and that $f \not= 0$ (if $f = 0$ then we can replace $f$ with an arbitrary function from $(c_0)_+ \setminus \{0\}$). Choose $\mu \in \sigma(T)$ with $|\mu| = 1$.
		
	For every $r > 1$ we define $\alpha(r) := \sum_{n=0}^\infty \frac{f_n}{r^{n+1}} > 0$. Using the Neumann series representation of the resolvent we obtain for all $x \in E_+$, $x' \in E'_+$ and $r \in (1,\infty)$ that
	\begin{align*}
		|\langle x', \Res(r \mu,T)x \rangle| - \re \langle x', & \Res(r,T) x \rangle \le \sum_{n=0}^\infty \frac{|\langle x', T^nx \rangle| - \re \langle x', T^nx \rangle}{r^{n+1}} \\
		& \le 2 \sum_{n=0}^\infty \frac{\distPos(\langle x', T^nx \rangle)}{r^{n+1}} \le 2 \sum_{n=0}^\infty \frac{\distPos(\langle x', T^nx \rangle)^*}{r^{n+1}},
	\end{align*}
	where we used Remark~\ref{rem:estimate-for-real-part-and-distance-to-positive-cone} to obtain the inequality between the first and the second line and where an infinite series version of the rearrangement inequality (see e.g.\ \cite[Lemma~3.3]{Glueck2015}) yields the inequality in the second line.
	
	By assumption there exists a number $c \ge 0$ (which might depend on $x$ and $x'$) such that $\distPos(\langle x', T^nx \rangle)^* \le cf_n$ for each $n \in \bbN_0$, so we conclude that 
	\begin{align}
		\label{eq:resolvent-estimate-ind-asymp-pos}
		\begin{aligned}
			\re \langle x', \Res(r,T)x \rangle & \ge |\langle x', \Res(r\mu,T)x \rangle| - 2c \, \sum_{n=0}^\infty \frac{f_n}{r^{n+1}} \\
			& = |\langle x', \Res(r\mu,T)x \rangle| - 2c \, \alpha(r)
		\end{aligned}
	\end{align}
	for all $r > 1$. 
	
	One can easily check that $(r-1)\alpha(r) \to 0$ for $r \downarrow 1$ since $f \in c_0$. Noting that $\|R(r\mu,T)\| \ge \frac{1}{\operatorname{dist}(r\mu, \sigma(T))} = \frac{1}{r-1}$ we thus conclude that $\lim_{r \downarrow 1} \frac{\|R(r\mu,T)\|}{\alpha(r)} = \infty$. Due to the uniform boundedness principle we can therefore find vectors $x \in E_+$ and $x' \in E'_+$ and a sequence of real numbers $r_k \downarrow 1$ such that
	\begin{align*}
		\lim_{k \to \infty} \frac{|\langle x', R(r_k\mu,T)x \rangle|}{\alpha(r_k)} = \infty \text{.}
	\end{align*}
	Using~\eqref{eq:resolvent-estimate-ind-asymp-pos} and the fact that $\liminf_{r \downarrow} \alpha(r) > 0$ we thus obtain the estimate
	\begin{align*}
		\re \langle x', R(r_k,T) x \rangle \ge \alpha(r_k) \Big( \frac{|\langle x', R(r_k\mu,T)x \rangle|}{\alpha(r_k)} - 2c \Big) \overset{k \to \infty}{\to} \infty \text{.}
	\end{align*}
	Hence, $\lim_{k \to \infty} \|R(r_k,T)\| = \infty$ which proves that $1 \in \sigma(T)$.
\end{proof}

In the following corollary we show that the conclusion of Theorem~\ref{thm:weakly-asymp-pos-with-rate-implies-spec-rad-in-spec} remains true if one allows $f$ to vary within a countable set. The proof is virtually the same as the proof of \cite[Proposition~3.4 and~Corollary~3.5]{Glueck2015}; for the convenience of the reader we include the entire argument here.

\begin{corollary} \label{cor:weakly-asymp-pos-with-countably-varying-rate-implies-spec-rad-in-spec}
	Let $E$ be a complex Banach lattice, let $T \in \calL(E)$ with $\spr(T) > 0$ and define $S := T/\spr(T)$. Let $F \subseteq (c_0)_+$ be an at most countable set and suppose that $F$ governs the sequence
	\begin{align*}
		\big(\distPos(\langle x', S^n x \rangle)\big)_{n \in \bbN_0}
	\end{align*}
	for each $x \in E_+$ and each $x \in E'_+$. Then $r(T) \in \sigma(T)$.
\end{corollary}
\begin{proof}
	It follows from the assumptions that $F \not= \emptyset$; we may assume that $F$ does not contain $0$. Let $(f_n)_{n \in \bbN}$ be an enumeration of the elements of $F$ (where some of the vectors in $F$ may occur several times in case that $F$ is finite). We define
	\begin{align*}
		f := \sum_{n=1}^\infty \frac{f_n}{2^n\|f_n\|} \in c_0.
	\end{align*}
	Then $f$ dominates a multiple of every element of $F$ and hence, $\{f\}$ governs the sequence
	\begin{align*}
		\big(\distPos(\langle x', S^n x \rangle)\big)_{n \in \bbN_0}
	\end{align*}
	for all $x \in E$ and all $x' \in E'$. According to Theorem~\ref{thm:weakly-asymp-pos-with-rate-implies-spec-rad-in-spec} this implies that $\spr(T) \in \spec(T)$.
\end{proof}

We close this section by demonstrating why Theorem~\ref{thm:countable-summability-condition-implies-spec-in-spec} is a consequence of Corollary~\ref{cor:weakly-asymp-pos-with-countably-varying-rate-implies-spec-rad-in-spec}:

\begin{proof}[Proof of Theorem~\ref{thm:countable-summability-condition-implies-spec-in-spec}]
	Let $A$ be a set of real-valued sequences $a = (a_n)_{n \in \bbN_0} \subseteq [0,\infty)$ and let $\varphi \in \Phi$. It was proved in \cite[Lemma~4.1]{Glueck2015} that if $\sum_{n=0} \varphi(a_n) < \infty$ for each $a \in A$, then there exists a sequence $0 \le f \in c_0$ such that $\{f\}$ governs each element of $A$.
	
	For every $\varphi \in \Phi$ we now define $D_\varphi$ to be the set of all pairs $(x,x') \in E \times E'$ which fulfil $\|x\| \le 1$ and $\|x'\| \le 1$ and for which $\sum_{n=0}^\infty \varphi\big(\distPos(\langle x', S^n x \rangle)\big) < \infty$. As noted at the beginning of the proof we can find, for each $\varphi \in \Phi$, a sequence $0 \le f_\varphi \in c_0$ such that $\{f_\varphi\}$ governs $\big(\distPos(\langle x', S^n x \rangle)\big)_{n \in \bbN_0}$ for all $(x',x) \in D_\varphi$. By assumption, we have
	\begin{align*}
		\bigcup_{\varphi \in \Phi} D_\varphi = \{(x,x') \in E' \times E: \; \|x\| \le 1 \text{ and } \|x'\| \le 1 \},
	\end{align*}
	so $F := \{f_\varphi: \; \varphi \in \Phi\}$ governs $\big(\distPos(\langle x', S^n x \rangle)\big)_{n \in \bbN_0}$ for all $x \in E$ and $x' \in E'$ of norm $\le 1$. By a simple scaling argument it follows that $F$ governs the sequence $\big(\distPos(\langle x', S^n x \rangle)\big)_{n \in \bbN_0}$ for actually all $x \in E$ and $x' \in E'$. Since $F$ is at most countable, it follows from Corollary~\ref{cor:weakly-asymp-pos-with-countably-varying-rate-implies-spec-rad-in-spec} that $\spr(T) \in \spec(T)$.
\end{proof}

\section{Positive Eigenvectors} \label{section:positive-eigenvectors}

In this section we give sufficient conditions for the spectral radius of an operator $T$ to be an eigenvalue of $T$ which admits a positive eigenvector. Even for a positive operator the spectral radius need, in general, not be an eigenvalue at all (the multiplication operator with symbol $(1-\frac{1}{n})$ on $\ell^p(\bbN;\bbC)$ is a counterexample). A very common condition to ensure that a spectral value $\lambda$ of an arbitrary operator $T$ is an eigenvalue is to assume that it is a pole of the resolvent $\Res(\phdot,T)$. This is also a common assumption in Perron--Frobenius theory. Under this assumption we obtain the following Kre\u{\i}n--Rutman type result.

\begin{theorem} \label{thm:positive-eigenvectors}
	Let $E$ be a complex Banach lattice and let $T \in \calL(E)$ with $\spr(T) > 0$. Assume that $T$ is weakly asymptotically positive and that $\spr(T)$ is a spectral value of $T$ and a pole of the resolvent $\Res(\phdot,T)$.
	
	Then $\spr(T)$ is an eigenvalue of $T$ and of the adjoint $T'$ and each of the eigenspaces $\ker(\spr(T) - T)$ and $\ker(\spr(T) - T')$ contains a non-zero positive vector.
\end{theorem}

Recall that sufficient conditions for $\spr(T)$ to be a spectral value of $T$ are given in Sections~\ref{section:the-spectral-radius-1} and~\ref{section:the-spectral-radius-2}. Suppose that $0 < \spr(T) \in \spec(T)$. The assumption that $\spr(T)$ be a pole of $\Res(\phdot,T)$ is for example fulfilled if the essential spectral radius $\spr_\ess(T)$ is strictly smaller than $\spr(T)$ (see e.g.\ \cite[formula~(1.16) on p.\,249]{Engel2000}). It is also fulfilled if there exists an open neighbourhood $U$ of $\spec(T)$ and an analytic function $f: U \to \bbC$ such that $f(T)$ is compact and such that $f(\spr(T)) \not= 0$ \cite[Theorem~5.8-F]{Taylor1958}. In particular, $\spr(T)$ is a pole of $\Res(\phdot,A)$ if $T$ or some power of $T$ is compact. Combining these observations with our results from Sections~\ref{section:the-spectral-radius-1} and~\ref{section:the-spectral-radius-2} we obtain, for instance, the following corollaries:

\begin{corollary}
	Let $E$ be a complex Banach lattice and let $T \in \calL(E)$ be a weakly eventually positive operator with $\spr(T) > 0$. Assume that some power of $T$ is compact or, more generally, that there exists an open neighbourhood $U$ of $\spec(T)$ and an analytic function $f: U \to \bbC$ with $f(\spr(T)) \not= 0$ for which $f(T)$ is compact. 
	
	Then $\spr(T)$ is an eigenvalue of $T$ and $T'$ and each of the eigenspaces $\ker(\spr(T) - T)$ and $\ker(\spr(T) - T')$ contains a non-zero positive vector.
\end{corollary}
\begin{proof}
	It follows from Corollary~\ref{cor:weakly-ev-pos-implies-spec-in-spec} that $\spr(T)$ is a spectral value of $T$. Moreover, as recalled above, $\spr(T)$ is a pole of the resolvent $\Res(\phdot,T)$. Hence, the assertion follows from Theorem~\ref{thm:positive-eigenvectors}.
\end{proof}

\begin{corollary}
	Let $E$ be a complex Banach lattice and let $T \in \calL(E)$ be an operator which fulfils $0 \le \spr_\ess(T) < \spr(T)$ and which is weakly asymptotically positive. Suppose that $T/\spr(T)$ is power bounded.
	
	Then $T$ is uniformly asymptotically positive; moreover, $\spr(T)$ is an eigenvalue of $T$ and $T'$ and each of the eigenspaces $\ker(\spr(T) - T)$ and $\ker(\spr(T) - T')$ contains a non-zero positive vector.
\end{corollary}
\begin{proof}
	Since $T/\spr(T)$ is power bounded and since $\spr_\ess(T) < \spr(T)$, it is easy to see that the set $\{(T/\spr(T))^n: \; n \in \bbN_0\}$ is relatively compact in $\calL(E)$ with respect to the operator norm topology. Hence, it follows from Proposition~\ref{prop:regularity-implies-better-asymp-pos}(b) that $T$ is uniformly asymptotically positive. 
	
	Thus, the spectral radius $\spr(T)$ is contained in $\spec(T)$ according to Theorem~\ref{thm:spec-rad-in-spec-for-unif-asymp-pos-op}. Since $\spr_\ess(T) < \spr(T)$ we know that $\spr(T)$ is a pole of the resolvent $\Res(\phdot,T)$, so the assertion follows from Theorem~\ref{thm:positive-eigenvectors}.
\end{proof}

The proof of Theorem~\ref{thm:positive-eigenvectors} uses some well-known properties of the Laurent series expansion of the resolvent and is quite elementary:

\begin{proof}[Proof of Theorem~\ref{thm:positive-eigenvectors}]
	We may assume that $\spr(T) = 1$. Let us begin with a preliminary observation. For all $x \in E_+$, all $x' \in E'_+$ and all $r > 1$ we have
	\begin{align*}
		\distPos(\langle x', \Res(r,T)x \rangle) \le \sum_{n=0}^\infty \frac{\distPos(\langle x', T^nx \rangle)}{r^{n+1}}
	\end{align*}
	according to the Neumann series representation of the resolvent. Using that we have $\distPos(\langle x', T^nx \rangle) \to 0$ as $n \to \infty$ and that $(r-1)\sum_{n=0}^\infty \frac{1}{r^{n+1}} = 1$ for all $r > 1$, we thus obtain
	\begin{align}
		\label{eq:resolvent-weakly-asymptotically-positive}
		(r-1)\distPos(\langle x', \Res(r,T)x \rangle) \to 0 \qquad \text{as} \qquad r \downarrow 1.
	\end{align} 
	Now, let $m \in \bbN$ denote the order of $1$ as a pole of the resolvent $\Res(\phdot,T)$ and let
	\begin{align*}
		\Res(\lambda,T) = \sum_{n=-m}^\infty Q_n(\lambda - 1)^n
	\end{align*}
	be the Laurent series expansion of the resolvent about $1$ (where $Q_n \in \calL(E)$ for all $n \in \{-m,-m+1,...\}$). Note that the operator $Q_{-m}$ is non-zero and that its range $Q_{-m}E$ is contained in $\ker(1 - T)$ \cite[Theorem~2 in Section~VIII.8]{Yosida1995}. In particular, $1$ is an eigenvalue of $T$. 
	
	Moreover, the operator $Q_{-m}$ is positive: indeed, $(r-1)^m \Res(r,T)$ converges to $Q_{-m}$ with respect to the operator norm as $r \downarrow 1$, so it follows from~\eqref{eq:resolvent-weakly-asymptotically-positive} that $\distPos(\langle x', Q_{-m}x \rangle) = 0$ for all $x \in E_+$ and all $x' \in E'_+$. Since $Q_{-m}$ is positive and non-zero and since $E = E_+ - E_+$, we can find a vector $x \in E_+$ such that $Q_{-m}x > 0$. Thus, $Q_{-m}x$ is a positive eigenvector of $T$ for the eigenvalue $1$.
	
	Let us now consider the adjoint operator $T'$. It has the same spectrum as $T$ and we have $\Res(\lambda,T') = \Res(\lambda,T)'$ for all $\lambda \in \spec(T') = \spec(T)$. The Laurent series representation of $\Res(\phdot,T')$ about $1$ is thus given by
	\begin{align*}
		\Res(\lambda,T') = \sum_{n=-m}^\infty Q_n' (\lambda - 1)^n.
	\end{align*}
	Since $Q_{-m}' \not= 0$, it follows that $1$ is also a pole of order $m$ of $\Res(\phdot,T')$. As above we have $Q_{-m}'E' \subseteq \ker(1 - T')$. Since $Q_{-m}$ is positive, so is $Q_{-m}'$ and hence, there exists a vector $x' \in E'_+$ such that $Q_{-m}' x' > 0$. Thus, $Q_{-m}'x'$ is a positive eigenvector of $T$ for the eigenvalue $1$.
\end{proof}

\section{The Peripheral Spectrum} \label{section:the-peripheral-spectrum}

In this section we turn to the peripheral spectrum $\spec_\per(T)$ of an asymptotically positive operator $T$. We note once again that $\spec_\per(T)$ is defined to be the set of all spectral values of $T$ with maximal modulus.

Recall that an operator $T$ on a complex Banach space $E$ is called \emph{Abel bounded} if $\sup_{\lambda > \spr(T)} (\lambda - \spr(T))\|\Res(\lambda,T)\| < \infty$. An easy application of the Neumann series representation of the resolvent shows that an operator $T$ with non-zero spectral radius is automatically Abel bounded in case that $T/\spr(T)$ is power bounded. The converse implication is, however, not true: for instance, an operator is automatically Abel bounded if $\spr(T) \not\in \spec(T)$; moreover, there exist even positive operators (for which we always have $\spr(T) \in \spec(T)$) which are Abel bounded, but for which the rescaled operator $T/\spr(T)$ is not power bounded; see \cite[Section~2]{Derriennic1973} for a counterexample.

A deep result in Perron--Frobenius theory asserts that the peripheral spectrum of a positive, Abel-bounded operator $T$ on a complex Banach lattice is automatically cyclic, i.e.\ we have $\spr(T)e^{in\theta} \in \spec(T)$ for all integers $n \in \bbZ$ whenever $\spr(T)e^{i\theta} \in \spec(T)$ ($\theta \in \bbR$). This was proved independently by Krieger \cite[Folgerung 2.2.1(b)]{Krieger1969} and Lotz \cite[Theorem~4.7]{Lotz1968} in the late 1960s. It is worthwhile pointing out that the question whether the peripheral spectrum of \emph{every} positive operator on a complex Banach lattice is cyclic is an open problem until today; we refer to \cite{Glueck2016} and \cite{GlueckGR} for a detailed discussion of this topic and for some recent partial results.

Here we show that the peripheral spectrum of a uniformly \emph{asymptotically} positive operator $T$ is cyclic in case that $T/\spr(T)$ is power bounded:

\begin{theorem} \label{thm:unif-asympt-pos-adjoint-with-bdd-error-cyclcic-per-spec}
	Let $E$ be a complex Banach lattice and let $T \in \calL(E)$ with $\spr(T) > 0$. If $T/\spr(T)$ is power-bounded and $T$ is uniformly asymptotically positive, then $\spec_\per(T)$ is cyclic.
\end{theorem}

We have not been able yet to prove or disprove the same assertion for operators which are merely Abel bounded. Similarly to Krieger and Lotz we employ some kind of lifting technique to transform the peripheral spectrum of an operator into point spectrum (more precisely, we use ultra powers of Banach lattices). The rest of our proof is, however, quite different from the arguments used by Krieger and Lotz. 

Let us give a very brief reminder of ultra powers of Banach lattices. Let $E$ be a complex Banach lattice and let $\calU$ be a free ultra filter on $\bbN$. By $\ell^\infty(\bbN;E)$ we denote the space of all $E$-valued norm bounded sequences, endowed with the supremum norm; note that $\ell^\infty(\bbN;E)$ is itself a complex Banach lattice. By $c_{0,\calU}(\bbN;E)$ we denote the closed ideal in $\ell^\infty(\bbN;E)$ of all sequences which converge to $0$ along $\calU$. The quotient space
\begin{align*}
	E^\calU := \ell^\infty(\bbN;E) / c_{0,\calU}(\bbN;E)
\end{align*}
is called an \emph{ultra power} of $E$; it is itself a complex Banach lattice. For every sequence $x = (x_n) \in \ell^\infty(\bbN;E)$ we denote by $x^\calU := (x_n)^\calU$ the equivalence class of $x$ in $E^\calU$; it is not difficult to see that the norm of $x^\calU$ in $E^\calU$ is given by $\|x^\calU\| = \lim_\calU \|x_n\|$. Moreover, we have 
\begin{align}
	\label{eq:distance-to-positive-cone-in-ultra-power}
	\distPos(x^\calU) = \lim_\calU \distPos(x_n)
\end{align}
for all $x^\calU = (x_n)^\calU \in E^\calU$; this follows from Proposition~\ref{prop:formula-for-distance-to-positive-cone}. For every $x \in E$ we denote by $x^\calU$ the equivalence class of the constant sequence $(x,x,...)$ in $E^\calU$. Note that the mapping $E \to E^\calU$, $x \mapsto x^\calU$ is an isometric lattice homomorphism.

Now, let $T \in \calL(E)$. Then we define $T^\calU \in \calL(E^\calU)$ to be the operator given by $T^\calU x^\calU = (Tx_n)^\calU$ for all $x^\calU \in E^\calU$. The mapping $\calL(E) \to \calL(E^\calU)$, $T \mapsto T^\calU$ is an isometric Banach lattice homomorphism, and $T^\calU$ is positive if and only if $T$ is positive. Moreover, we have $\spec(T) = \spec(T^\calU)$ and $\spec_\pnt(T^\calU) = \spec_\appr(T^\calU) = \spec_\appr(T)$; in particular, the peripheral spectrum of $T^\calU$ consists of eigenvalues of $T^\calU$ and coincides with the peripheral spectrum of $T$. For more details we refer to \cite[Section~V.1]{Schaefer1974}, \cite[pp.\,251--253]{Meyer-Nieberg1991} and to the survey article \cite{Heinrich1980}.

In order to give the proof of Theorem~\ref{thm:unif-asympt-pos-adjoint-with-bdd-error-cyclcic-per-spec} we need one further ingredient, namely the next proposition. Let $E$ be a complex Banach lattice, which is by definition the complexification of a real Banach lattice $E_\bbR$, and let $F \subseteq E$ be a closed vector subspace. We call $F$ a \emph{lattice subspace} of $E$ and the \emph{real part} $F_\bbR := F \cap E_\bbR$ of $F$ fulfils $F_\bbR + iF_\bbR = E_\bbR$ and if $F_\bbR$ is a vector lattice with respect to the order induced by $E_\bbR$. We also recall that the dual space $E'$ of a complex Banach lattice $E$ is itself a complex Banach lattice; more precisely, if $E$ is a complexification of a real Banach lattice $E_\bbR$, then $E'$ is a complexification of $E_\bbR'$ \cite[Corollary~3 to Theorem~IV.1.8]{Schaefer1974}.

\begin{proposition} \label{prop:fixed-space-of-dual-operator}
	Let $E$ be a complex Banach lattice and let $T \in \calL(E)$ be positive and power bounded. Then the fixed space $F := \ker(1-T')$ of the adjoint operator $T'$ is a lattice subspace of $E'$; moreover, there exists a norm on $F$ which is equivalent to the norm induced by $E'$ and which renders $F$ a complex Banach lattice.
\end{proposition}
\begin{proof}
	By definition, $E$ is the complexification of a real Banach lattice $E_\bbR$; the space $E'$ is the complexification of the dual Banach lattice $E_\bbR'$. We define $F_\bbR := F \cap E_\bbR'$; since $T'$ maps $E_\bbR'$ to $E_\bbR'$ we clearly have $F_\bbR + iF_\bbR = F$. In order to prove that $F$ is a lattice subspace of $E'$ we thus have to show that $F_\bbR$ is a vector lattice with respect to the order induced by $E_\bbR'$. To this end, it suffices to proves that, for every $f \in F_\bbR'$, there exists a supremum of $f$ and $-f$ in $F_\bbR$.
	
	So let $f \in F_\bbR$. We have $T'f = f$ and thus $T'|f| \ge |f|$. Iterating this inequality we obtain that the sequence $((T')^n|f|)_{n \in \bbN_0}$ is increasing. Note that the sequence is also norm bounded since we assumed $T$ to be power bounded. Hence, $((T')^n|f|)_{n \in \bbN_0}$ converges to a vector $0 \le g \in E_\bbR'$ with respect to the weak${}^*$-topology. Since $T'$ is continuous with respect to this topology, $g$ is a fixed point of $T'$ and thus contained in $F_\bbR$. Since $g \ge |f|$, the vector $g$ is clearly an upper bound of $f$ and $-f$ in $F_\bbR$. Assume now, on the other hand, that $h \in F_\bbR$ is another upper bound of $f$ and $-f$ in $F_\bbR$. Then we have $|f| \le h$ and hence $(T')^n|f| \le (T')^n h = h$ for all $n \in \bbN_0$. This proves that $g \le h$, so $g$ is indeed the supremum of $f$ and $-f$ in $F_\bbR$. We have thus proved that $F$ is indeed a lattice subspace of $E'$.
	
	Finally, we denote the modulus of any $f \in F_\bbR$ in the vector lattice $F_\bbR$ by $|f|_F$ and we define $\|f\|_F := \| |f|_F\|$ for every $f \in F_\bbR$. We clearly have $\|f\| \le \|f\|_F$ for all $f \in F_\bbR$ and from the construction of $|f|_F$ in the above part of the proof we obtain that $\|f\|_F \le \sup_{n \in \bbN_0} \|T^n\| \|f\|$. Hence, the norm $\|\cdot\|_F$ on $F_\bbR$ is equivalent to the norm $\|\cdot\|$ induced by $E'$. It is now straightforward to check that $(F_\bbR, \|\cdot\|_F)$ is a (real) Banach lattice, and from this it readily follows that $F = F_\bbR + iF_\bbR$ is a complex Banach lattice with respect to a norm equivalent to the norm induced by $E'$.
\end{proof}

Arguments as used in the above proof are quite common in Perron--Frobenius theory, compare for instance \cite[the proof of Corollary~C-III.4.3(a)]{Arendt1986}. We also refer to \cite[Theorem~2.1 and Corollary~2.2]{Glueck2016} for related results.

The following proof of Theorem~\ref{thm:unif-asympt-pos-adjoint-with-bdd-error-cyclcic-per-spec} exhibits some similarities to the proof of \cite[Theorem~3.2]{Glueck2016}; the technical details are, however, quite different.

\begin{proof}[Proof of Theorem~\ref{thm:unif-asympt-pos-adjoint-with-bdd-error-cyclcic-per-spec}]
	We may assume that $\spr(T) = 1$. Let $\lambda \in \spec_\per(T)$, i.e.\ let $\lambda \in \spec(T)$ and $|\lambda| = 1$. We have to prove that $\lambda^m \in \spec(T)$ for all $m \in \bbZ$. Replacing $E$ with an ultra power and $T$ with its lifting to this ultra power we may assume that $\lambda$ is an eigenvalue of $T$ with an eigenvector $z \in E$ (note that the uniform asymptotic positivity of $T$ is conserved if we lift $T$ to an ultra power of $E$; this follows from formula~\eqref{eq:distance-to-positive-cone-in-ultra-power}).
	
	Let us now employ a second ultra power argument. Choose a free ultra filter $\calU$ on $\bbN$ and a sequence of integers $1 \le k_n \to \infty$ such that $\lambda^{k_n} \to 1$ (such a sequence clearly exists). We define two operators $R$ and $S$ on the ultra power $E^\calU$ which are given by $Rx^\calU = (T^{k_n-1}x_n)^\calU$ and $Sx^\calU = (T^{k_n}x_n)^{\calU}$ for all $x^\calU \in E^\calU$. Using that $T$ is power bounded it is easy to see that $R$ and $S$ are well-defined. Moreover, the operators $T^\calU$, $R$ and $S$ commute, they are power bounded and we have $R T^\calU = T^\calU R = S$. Since $T$ is uniformly asymptotically positive, it follows from formula~\eqref{eq:distance-to-positive-cone-in-ultra-power} that $R$, $S$ and $S T^\calU = T^\calU S$ are positive operators on the complex Banach lattice $E^\calU$. Note that the vector $z^\calU$ is contained in the fixed space $\ker(1-S)$ of $S$ since $\lambda^{k_n} \to 1$.
	
	We now take bi-adjoints; to keep the notation simple, let us define $\hat E := (E^\calU)''$, $\hat T := (T^\calU)''$, $\hat R := R''$ and $\hat S := S''$. Then the operators $\hat T$, $\hat R$ and $\hat S$ commute, they are power bounded and we have $\hat R \hat T = \hat T \hat R = \hat S$; moreover, $\hat R$, $\hat S$ and $\hat S \hat T = \hat T \hat S$ are positive. Proposition~\ref{prop:fixed-space-of-dual-operator} shows that the fixed space $F := \ker(1-\hat S)$ of $\hat S$ is a lattice subspace of $\hat E$ and a complex Banach lattice with respect to an equivalent norm. Since $\hat T$ and $\hat R$ commute with $\hat S$, they leave $F$ invariant, and their restrictions to $F$ fulfil $\hat R|_F \hat T|_F = \hat T|_F \hat R_F = \hat S|_F = \id_F$; hence, $\hat T|_F$ and $\hat R|_F$ are inverse to each other. Since $\hat R$ is positive, so is $\hat R|_F$, and since $\hat T \hat S$ is positive, so is $\hat T|_F$. This proves that $\hat T_F$ is actually a Banach lattice isomorphism on the complex Banach lattice $F$. 
	
	We consider $E^\calU$ as a subspace of $\hat E$ by means of evaluation. Since the eigenvector $z^\calU$ of $T^\calU$ for the eigenvalue $\lambda$ is contained in $\ker(1-S)$, it is also contained in $F = \ker(1-\hat S)$; so $\lambda$ is an eigenvalue of $\hat T|_F$. Since the point spectrum of every lattice homomorphism on a complex Banach lattice is cyclic \cite[Corollary~2 to Proposition~V.4.2]{Schaefer1974}, it follows that $\lambda^m$ is an eigenvalue of $\hat T|_F$, and thus of $\hat T$, for every $m \in \bbZ$. Therefore, $\lambda^m$ is a spectral value of $T^\calU$, and hence of $T$, for every $m \in \bbZ$.
\end{proof}

We close this section with a few comments on the assumptions of Theorem~\ref{thm:unif-asympt-pos-adjoint-with-bdd-error-cyclcic-per-spec}:

\begin{remarks}
	(a) The assertion of Theorem~\ref{thm:unif-asympt-pos-adjoint-with-bdd-error-cyclcic-per-spec} does not in general remain true if $T$ is only assumed to be individually asymptotically positive instead of uniformly asymptotically positive. A counterexample is again provided by the multiplication operator on $\ell^p(\bbN;\bbC)$ ($1 \le p < \infty$) with symbol $(-1+\frac{1}{n})_{n \in \bbN}$.
	
	(b) It does not seem to be clear whether the peripheral spectrum of an \emph{individually eventually} positive operator $T$ is cyclic (in case that $T/\spr(T)$ is power bounded). One might conjecture that every individually eventually positive operator $T$ is uniformly asymptotically positive (at least if $T/\spr(T)$ is power bounded), in which case the answer to this question would be positive due to Theorem~\ref{thm:unif-asympt-pos-adjoint-with-bdd-error-cyclcic-per-spec}; yet, it does not seem to be clear either whether such a conjecture is justified (compare the discussion before Theorem~\ref{thm:countable-summability-condition-implies-spec-in-spec}). 
\end{remarks}

\section{The Peripheral Point Spectrum} \label{section:the-periphera-point-spectrum}

In this final section we consider the peripheral \emph{point} spectrum rather than the peripheral spectrum. Recall that the peripheral point spectrum $\spec_{\per,\pnt}(T)$ of an operator $T$ is defined to be the set of all eigenvalues of $T$ with modulus $\spr(T)$. We point out that the peripheral point spectrum can be empty, in general.

While the peripheral spectrum of a positive operator $T$ is always cyclic in case that $T/\spr(T)$ is power bounded, this is not in general true for the peripheral point spectrum; see for instance \cite[Sections~5 and~6]{Glueck2016} where several counterexamples are discussed. On the other hand, the same reference contains many sufficient conditions which ensure that the peripheral point spectrum of a positive operator is indeed cyclic. Here, we adapt one of these conditions (namely \cite[Theorem~5.5]{Glueck2016}) to the case of weakly asymptotically positive operators. 

Recall that an operator $S$ on a Banach space $E$ is said to have \emph{relatively weakly compact orbits} if the set $\{S^n: \; n \in \bbN_0\}$ is relatively compact in $\calL(E)$ with respect to the weak operator topology. This is equivalent to the set $\{S^nx: \; n \in \bbN_0\}$ being relatively compact in $E$ with respect to the weak topology for each $x \in E$ (see e.g.\ \cite[Corollary~A.5]{Engel2000}). Note that every power bounded operator on a reflexive Banach space automatically has relatively weakly compact orbits.

\begin{theorem} \label{thm:criterion-for-cyclic-per-point-spec}
	Let $E$ be a complex Banach lattice and let $T \in \calL(E)$ with $\spr(T) > 0$. Suppose that the powers of $T/\spr(T)$ are relatively weakly compact and that $T$ is weakly asymptotically positive. Then we have
	\begin{align*}
		\dim \ker (\spr(T) e^{i\theta} - T) \le \dim \ker (\spr(T) e^{in \theta} - T)
	\end{align*}
	for all $n \in \bbZ$ and all $\theta \in \bbR$. In particular, the peripheral point spectrum of $T$ is cyclic.
\end{theorem}

In the above theorem we understand the dimension of a vector space to be either an integer or $\infty$, i.e.\ we do not distinguish between different infinite cardinalities. The proof of the above theorem relies on the so-called \emph{Jacobs--de Leeuw--Glicksberg decomposition} of an operator (see for instance \cite[Section~2.4]{Krengel1985}, \cite[Section~V.2]{Engel2000} or \cite[Section~16.3]{Eisner2015} for details about this construction) and is very similar to the proof of \cite[Theorem~5.5]{Glueck2016}. The only difference is that we now use a bit more information about the Jacobs--de Leeuw--Glicksberg decomposition to compensate for the fact that the operator might no longer be positive, but only weakly asymptotically positive.

\begin{proof}[Proof of Theorem~\ref{thm:criterion-for-cyclic-per-point-spec}]
	We may assume that $\spr(T) = 1$. Let $\calS$ denote the closure of the set $\{T^n: \; n \in \bbN_0\}$ in $\calL(E)$ with respect to the weak operator topology. Then $\calS$ is a compact commutative semi-topological semigroup with respect to operator multiplication and with respect to the weak operator topology. Hence, the so-called \emph{Sushkevich kernel}
	\begin{align*}
		\calK := \bigcap_{S \in \calS} S \calS
	\end{align*}
	is an ideal in the semigroup $\calS$ and in fact it is even a group. The neutral element $P$ of the group $\calK$ is a projection on $E$; since $P$ commutes with every operator $S \in \calS$ it reduces each such $S$. The restriction of $T$ to the range $PE$ of $P$ is an invertible operator in $\calL(PE)$ and its inverse $(T|_{PE})^{-1}$ is given by $R|_{PE}$ for some operator $R \in \calS$. Moreover, the range of $P$ is the closed linear span of all eigenvectors of $T$ belonging to eigenvalues of modulus $1$; in particular, we have $\ker(e^{i\theta}- T) = \ker(e^{i\theta} - T|_{PE})$ for all $\theta \in \bbR$. All these results can, for instance, be found in \cite[Section~2.4]{Krengel1985}. 
	
	Now we prove that each $K \in \calK$ is a positive operator on $E$. To this end, we first note that we have $S\calS = \overline{S \{T^n: \; n \in \bbN_0\} }^{\operatorname{w}}$ for each $S \in \calS$ (where $\overline{\calA}^{\operatorname{w}}$ denotes the closure of any subset $\calA \subseteq \calL(E)$ with respect to the weak operator topology). Indeed, the inclusion ``$\subseteq$'' is obvious and the converse inclusion ``$\supseteq$'' follows from the weak compactness of $\calS$. Hence, we obtain
	\begin{align}
		\label{eq:sushkevich-kernel-asymptotic-behaviour}
		\calK = \bigcap_{S \in \calS} \overline{S \{T^n: \; n \in \bbN_0\} }^{\operatorname{w}} \subseteq \bigcap_{m \in \bbN_0} \overline{ \{T^{m+n}: \; n \in \bbN_0\} }^{\operatorname{w}}.
	\end{align}
	Let $K \in \calK$, let $x \in E_+$, $x' \in E'_+$ and let $\varepsilon > 0$. Since $T$ is weakly asymptotically positive, there exists an $m \in \bbN_0$ such that $\distPos(\langle x', T^{m+n}x \rangle) < \varepsilon$ for all $n \in \bbN_0$. Moreover, according to~\eqref{eq:sushkevich-kernel-asymptotic-behaviour}, we can find an integer $n \in \bbN_0$ such that
	\begin{align*}
		|\langle x', T^{m+n}x\rangle - \langle x', Kx\rangle| < \varepsilon.
	\end{align*}
	Hence, $\distPos(\langle x', Kx\rangle) < 2\varepsilon$. Since $\varepsilon > 0$ was arbitrary, it follows that $K$ is indeed positive.
	
	This implies in particular that the projection $P$ is positive, so its range is a lattice subspace of $E$ and a complex Banach lattice with respect to some equivalent norm (this is a simple consequence of the same result for positive projections on \emph{real} Banach lattices which can for instance be found in \cite[Proposition~III.11.5]{Schaefer1974}). We have $T|_{PE} = (TP)|_{PE}$; since the operator $TP$ is contained in $\calK$, it is positive and hence, so is $T|_{PE}$. Finally, recall that the inverse $(T|_{PE})^{-1}$ is given by $R|_{PE}$ for some $R \in \calK$, so it is also positive. This proves that the restricted operator $T|_{PE}$ is a lattice isomorphism on the complex Banach lattice $PE$. We now conclude for all $n \in \bbZ$ and all $\theta \in \bbR$ that
	\begin{align*}
		\dim \ker (e^{i\theta} - T) & = \dim \ker (e^{i\theta} - T|_{PE}) \\
		& \le \dim \ker(e^{in\theta} - T|_{PE}) = \dim \ker(e^{in\theta} - T);
	\end{align*}
	the dimension estimate between the first and the second line is true for every lattice homomorphism as proved in \cite[Proposition~3.1(b)]{Glueck2016}.
\end{proof}

\subsection*{Acknowledgements}

While most of the work on this article was done, the author was supported by a scholarship of the ``Landesgraduiertenf\"orderung Baden--W\"urttemberg'', Germany (grant number 1301 LGFG-E).

\bibliographystyle{plain}
\bibliography{literature}

\end{document}